\documentclass[11pt]{article}

\usepackage{subfig}
\usepackage{mathrsfs}
\usepackage{makecell}
\usepackage{amscd}
\usepackage{amsmath}
\usepackage{amssymb}
\usepackage{amstext}
\usepackage[thmmarks,amsthm,amsmath]{ntheorem}
\usepackage{apptools}
\usepackage{bbold}
\usepackage{bm}
\usepackage{booktabs}
\usepackage{color}
\usepackage{easybmat}
\usepackage{etex}
\usepackage{framed}
\usepackage[dvips,letterpaper,margin=1in]{geometry}
\usepackage{graphicx}
\usepackage{hyperref}
\usepackage[noabbrev,capitalize]{cleveref}
\usepackage{mathtools}
\usepackage{colonequals}
\usepackage{longtable}
\usepackage{rotating}
\usepackage{setspace}
\usepackage{tabu}
\usepackage{verbatim}

\newtheorem{theorem}{Theorem}[section]
\newtheorem*{remark}{Remark}
\newtheorem{lemma}[theorem]{Lemma}
\newtheorem{question}[theorem]{Question}

\newtheorem{proposition}[theorem]{Proposition}
\newtheorem{definition}[theorem]{Definition}
\newtheorem{corollary}[theorem]{Corollary}
\AtAppendix{\renewtheorem{corollary}[theorem]{Corollary}}

\crefname{theorem}{Theorem}{Theorems}
\crefname{proposition}{Proposition}{Propositions}
\crefname{lemma}{Lemma}{Lemmas}

\theoremstyle{plain} 
\newcommand{\thistheoremname}{}
\newtheorem*{genericthm}{\thistheoremname}

\newcommand{\what}{\widehat}

\DeclareSymbolFont{bbold}{U}{bbold}{m}{n}
\DeclareSymbolFontAlphabet{\mathbbold}{bbold}

\newcommand{\ra}{\rangle}
\newcommand{\la}{\langle}

\makeatletter
\def\moverlay{\mathpalette\mov@rlay}
\def\mov@rlay#1#2{\leavevmode\vtop{%
   \baselineskip\z@skip \lineskiplimit-\maxdimen
   \ialign{\hfil$\m@th#1##$\hfil\cr#2\crcr}}}
\newcommand{\charfusion}[3][\mathord]{
    #1{\ifx#1\mathop\vphantom{#2}\fi
        \mathpalette\mov@rlay{#2\cr#3}
      }
    \ifx#1\mathop\expandafter\displaylimits\fi}
\makeatother

\renewcommand{\dim}{\mathsf{dim}}

\newcommand{\RR}{\mathbb{R}}

\newcommand{\ZZ}{\mathbb{Z}}

\newcommand{\CC}{\mathbb{C}}

\newcommand{\FF}{\mathbb{F}}
\newcommand{\MM}{\mathbb{M}}

\newcommand{\One}{\mathbbold{1}}

\newcommand{\one}{\bm{1}}

\newcommand{\sB}{\mathcal{B}}

\newcommand{\sP}{\mathcal{P}}

\newcommand{\Tr}{\mathrm{Tr}}

\renewcommand{\ker}{\mathsf{ker}}

\renewcommand{\vec}{\mathsf{vec}}

\newcommand{\row}{\mathsf{row}}

\newcommand{\conv}{\mathsf{conv}}

\newcommand{\srg}{\mathsf{srg}}
\newcommand{\sym}{\mathsf{sym}}
\newcommand{\herm}{\mathsf{herm}}

\DeclareMathOperator*{\argmin}{\arg\!\min}
\newcommand{\diag}{\mathsf{diag}}

\newcommand{\rank}{\mathsf{rank}}

\newcommand{\bA}{\bm A}
\newcommand{\bB}{\bm B}

\newcommand{\bH}{\bm H}

\newcommand{\bN}{\bm N}
\newcommand{\bP}{\bm P}

\newcommand{\bR}{\bm R}

\newcommand{\bU}{\bm U}
\newcommand{\bV}{\bm V}

\newcommand{\bX}{\bm X}
\newcommand{\bY}{\bm Y}

\newcommand{\ba}{\bm a}
\newcommand{\bb}{\bm b}

\newcommand{\bv}{\bm v}

\newcommand{\bx}{\bm x}
\newcommand{\by}{\bm y}

\newcommand{\fC}{\mathscr{C}}
\newcommand{\fE}{\mathscr{E}}

\newcommand{\pert}{\mathsf{pert}}

\title{Sum-of-Squares Optimization and the Sparsity Structure of Equiangular Tight Frames}
\author{%
  Afonso S.\ Bandeira\footnote{Courant Institute of Mathematical Sciences and Center for Data Science, New York University, NY 10012. Afonso S. Bandeira was partially supported by NSF
grants DMS-1712730 and DMS-1719545, and by a grant from the Sloan
Foundation.} \and Dmitriy Kunisky\footnote{Courant Institute of Mathematical Sciences, New York University, NY 10012. Dmitriy Kunisky was partially supported by NSF
grants DMS-1712730 and DMS-1719545.}
}

\begin{document}

\maketitle

\begin{abstract}
    Equiangular tight frames (ETFs) may be used to construct examples of feasible points for semidefinite programs arising in sum-of-squares (SOS) optimization.
    We show how generalizing the calculations in a recent work of the authors' that explored this connection also yields new bounds on the sparsity of (both real and complex) ETFs.
    One corollary shows that Steiner ETFs corresponding to finite projective planes are optimally sparse in the sense of achieving tightness in a matrix inequality controlling overlaps between sparsity patterns of distinct rows of the synthesis matrix.
    We also formulate several natural open problems concerning further generalizations of our technique.
\end{abstract}

\newpage

\section{Introduction}

One of the most important objects in combinatorial optimization is the \emph{cut polytope}, the convex set of matrices
\begin{equation}
    \fC^N \colonequals \conv\left(\left\{ \bx\bx^\top: \bx \in \{ \pm 1\}^N\right\}\right) \subset \RR^{N \times N}_{\sym}.
\end{equation}
The cut polytope has a rich discrete geometry \cite{deza:09:book} describing the  solution space of the problem of finding the largest cut in a graph, which amounts to maximizing a linear function over $\fC^N$.
By the classical result of Karp \cite{karp:72}, this problem cannot be solved in polynomial time unless $\mathsf{P} = \mathsf{NP}$.
Therefore, \emph{relaxations} of $\fC^N$ to larger but more algorithmically tractable convex sets have been proposed, perhaps the best-known of which is the relaxation to the \emph{(real) elliptope}
\begin{equation}
    \fE^N \colonequals \left\{\bX \in \RR^{N \times N}_{\sym}: \bX \succeq \bm 0, \diag(\bX) = \one\right\} \supseteq \fC^N.
\end{equation}
There is extensive literature both on the geometry of $\fE^N$ (thoroughly described in \cite{deza:09:book}), and on approximating optimization over $\fC^N$ by optimization over $\fE^N$ (e.g.\ \cite{goemans:95,nesterov:98,feige:02,khot:07}).

Every $\bX \in \fE^N$ is the Gram matrix of unit vectors, $X_{ij} = \la \bv_i, \bv_j \ra$ for some $\bv_1, \dots, \bv_N \in \RR^{r}$ with $r \colonequals \rank(\bX) \leq N$.
The boundary of $\fE^N$ consists of $\bX$ having $r < N$.
In the language of frame theory, these boundary points are the Gram matrices of \emph{overcomplete unit norm frames}.
We will explore applications of geometric results about $\fE^N$ to the following types of structured frames.

\begin{definition}
    Unit vectors $\bv_1, \dots, \bv_N \in \CC^r$ with $\bX = (\la \bv_i, \bv_j \ra)_{i, j = 1}^N$ form a \emph{unit norm tight frame (UNTF)} if any of the following equivalent conditions hold.
    \begin{enumerate}
    \item $\sum_{i = 1}^N\bv_i\bv_i^* = \frac{N}{r} \bm I_N$.
    \item The nonzero eigenvalues of $\bX$ all equal $\frac{N}{r}$.
    \item $\|\bX\|_F^2 = \sum_{i = 1}^N \sum_{j = 1}^N |X_{ij}|^2 = \frac{N^2}{r}$.
    \end{enumerate}
    The $\bv_i$ form an \emph{equiangular tight frame (ETF)} if moreover there exists  $\alpha \in [0, 1]$ such that $|X_{ij}| = |\la \bv_i, \bv_j \ra| = \alpha$ whenever $i \neq j$.
\end{definition}
When moreover $\bv_i \in \RR^r$, we have $\bX \in \fE^N$, and such points form an interesting subset of the elliptope's boundary.
In both the real and complex cases, UNTFs and ETFs have been studied in great detail previously (e.g.\ \cite{benedetto:03,casazza:12,tropp:07,casazza:08,fickus:15:tables}).

The set $\fE^N$, however, is only the first of a sequence of tightening relaxations of $\fC^N$, described by the \emph{sum-of-squares (SOS)} hierarchy.
This gives sets we call \emph{degree $d$ generalized elliptopes} $\fE_d^N$ for positive even integers $d$.
\begin{definition}
    \label{def:pm-matrix}
    $\fE_d^N$ is the set of $\bX \in \RR^{N \times N}_{\sym}$ for which there exists $\bY \in \RR^{N^{d/2} \times N^{d / 2}}_{\sym}$ where, identifying indices of $\bY$ with elements of $[N]^{d / 2}$, the following conditions hold.
    \begin{enumerate}
    \item $Y_{\bm i \bm j} = 1$ whenever all indices occur an even number of times across $\bm i$ and $\bm j$.
    \item $Y_{\bm i \bm j}$ depends only on the set of indices occuring an odd number of times across $\bm i$ and $\bm j$.
    \item $Y_{(1, \dots, 1, i)(1, \dots, 1, j)} = X_{ij}$ for all $i, j \in [N]$.
    \item $\bY \succeq \bm 0$.
    \end{enumerate}
\end{definition}
We have $\fC^N \subseteq \fE_d^N$ since both sets are convex and $\bx\bx^\top \in \fE_d^N$ for $\bx \in \{ \pm 1\}^N$ by taking $\bY = (\bx^{\otimes d / 2})(\bx^{\otimes d / 2})^\top$.
The following inclusions also hold (these non-trivial facts are the combined results of \cite{fawzi:16} and \cite{laurent:03}):
\begin{equation}
    \fE^N = \fE_2^N \supsetneq \fE_4^N \supsetneq \cdots \supsetneq \fE_{N + \One\{N \text{ odd}\}}^N = \fC^N.
\end{equation}

In the recent work \cite{bandeira:18}, the authors showed that the Gram matrix of a real ETF belongs to $\fE_4^N$ if and only if $N < \frac{r(r + 1)}{2}$.
The construction involves the following general concept from convex geometry.
\begin{definition}
    For a closed convex set $K \subseteq \RR^d$ and $\bX \in K$, the \emph{perturbation of $\bX$ in $K$} is the subspace
    \begin{equation}
        \mathsf{pert}_K(\bX) \colonequals \left\{\bA \in \RR^d: \exists \hspace{0.3em} t > 0 \text{ with } \bX \pm t\bA \in K \right\}.
    \end{equation}
\end{definition}
We then showed in \cite{bandeira:18} that one may take
\begin{equation}
    \bY \colonequals \vec(\bX)\vec(\bX)^\top + \frac{N^2(1 - \frac{1}{r})}{\frac{r(r + 1)}{2} - N}\bP
\end{equation}
for $\bP$ the projection matrix to $\vec(\pert_{\fE_2^N}(\bX))$.
To verify the conditions of Definition~\ref{def:pm-matrix}, we computed $\bP$ for all real ETFs.

In this paper, we give a self-contained presentation of this calculation for perturbation subspaces in the complex generalization of the elliptope,
\begin{equation}
    \widetilde{\fE}^N \colonequals \left\{ \bX \in \CC^{N \times N}_{\herm}: \bX \succeq \bm 0, \diag(\bX) = \one\right\}.
\end{equation}
Using that this projection operator is positive semidefinite (psd), we derive inequalities in degree 4 polynomials of the entries of (real or complex) ETF vectors, which translate to new inequalities controlling the sparsity and spark of ETF vectors.

\subsection{Background}
\label{sec:background}

\subsubsection{Linear Algebra}

We first introduce some notations for standard linear-algebraic tools we will use.
Denote by $\CC^{N \times N}_{\herm}$ the set of $N \times H$ Hermitian matrices.
We recall the usual Hilbert space structure on $\CC^{N \times N}_{\herm}$.
\begin{definition}
  For $\bX, \bY \in \CC^{N \times N}_{\herm}$, define the \emph{Frobenius inner product} and associated \emph{Frobenius norm} as
  \begin{align}
    \la \bX, \bY \ra_{F} &\colonequals \Tr(\bX\bY) = \sum_{i = 1}^N\sum_{j = 1}^N X_{ij}^* Y_{ij} \in \RR, \\
    \|\bX\|_{F} &\colonequals \Tr(\bX^2)^{1/2} = \la \bX, \bX \ra_F^{1/2} = \left(\sum_{i = 1}^N \sum_{j = 1}^N |X_{ij}|^2\right)^{1/2} \geq 0.
  \end{align}
\end{definition}
We will also use the following entrywise transformations of matrices.
\begin{definition}
    For $\bA \in \CC^{a \times b}$, let $|\bA|$ denote the entrywise absolute value of $\bA$.
\end{definition}

\begin{definition}
    For $\bA \in \CC^{a \times b}$, let $\bA^{\odot k}$ denote the entrywise $k$th power of $\bA$ for $k \in \ZZ$.
\end{definition}

\subsubsection{Combinatorics}

\label{sec:background:comb}

We present some definitions of combinatorial objects involved in the construction of \emph{Steiner ETFs} that we will study in greater detail in the sequel (the construction itself is discussed in Section~\ref{sec:steiner}).
A full motivation of these definitions is beyond the scope of the present paper; we refer the reader to \cite{colbourn:06} for further information and references.

The following highly structured graphs are intimately related to general real ETFs, and will also play a role in analyzing our results for the special case of Steiner ETFs.
\begin{definition}
    A \emph{strongly regular graph with parameters $(v, k, \lambda, \mu)$}, abbreviated $\srg(v, k, \lambda, \mu)$, is a graph $G = (V, E)$ for which $|V| = v$, $G$ is $k$-regular, and every pair of distinct vertices in $G$ have $\lambda$ common neighbors if they are adjacent and $\mu$ common neighbors if they are not adjacent.
\end{definition}

The first type of combinatorial design we will be interested in for ETF constructions is the following generalization of the incidence structure of finite geometries.
\begin{definition}
    A \emph{$(t, k, v)$-Steiner system} is a pair $(S, \mathcal{B})$ where $|S| = v$ and $\mathcal{B} \subseteq 2^S$, such that for all $B \in \mathcal{B}$ we have $|B| = k$, and for any $C \subseteq S$ with $|C| = t$, there exists a unique $B \in \mathcal{B}$ such that $C \subseteq B$.
    The elements of $S$ are called \emph{points}, and the elements of $\mathcal{B}$ are called \emph{blocks}.
    A Steiner system is completely specified by its \emph{incidence matrix} $\bN \in \{0, 1\}^{v \times |\mathcal{B}|}$, with entries equal to 1 when the block index contains the point index and equal to 0 otherwise.
    Also associated to a Steiner system is its \emph{block intersection graph}, a graph on vertex set $\mathcal{B}$ where $B, B^\prime \in \mathcal{B}$ are adjacent if and only if $B \cap B^\prime \neq \emptyset$.
\end{definition}

\noindent
The following are basic combinatorial results on Steiner systems with $t = 2$.
\begin{proposition}
    Let $(S, B)$ be a $(2, k, v)$-Steiner system.
    Then, $b \colonequals |\mathcal{B}| = \frac{v(v - 1)}{k(k - 1)}$, and each point is contained in $\rho \colonequals \frac{v - 1}{k - 1}$ blocks (typically this is denoted $r$ in the combinatorics literature, which we adjust to avoid conflict with our notation for frame dimensions).
\end{proposition}

\noindent
Important examples of Steiner systems with $t = 2$ are given by finite projective and affine planes.
\begin{definition}
    A $(2, k, v)$-Steiner system $(S, \sB)$ is called a \emph{finite projective plane} if for any distinct $B, B^\prime \in \sB$ we have $|B \cap B^\prime| = 1$, and there exist $w, x, y, z \in S$ such that no $B \in \sB$ contains more than two of these points.
    The number $k - 1$ is called the \emph{order} of a finite projective plane.

    A $(2, k, v)$-Steiner system $(S, \sB)$ is called an \emph{finite affine plane} if $|B| \geq 2$ for all $B \in \sB$, for any $B \in \sB$ and $s \in S \setminus B$ there exists a unique $B^\prime \in \sB$ with $s \in B^\prime$ and $B \cap B^\prime = \emptyset$, and there exist $x, y, z \in S$ such that no $B \in \sB$ contains all three of these points.
    The number $k$ is called the \emph{order} of a finite affine plane.
\end{definition}

\begin{proposition}
    \label{prop:finite-planes}
    A finite projective plane of order $q$ is a $(2, q + 1, q^2 + q + 1)$-Steiner system.
    A finite affine plane of order $q$ is a $(2, q, q^2)$-Steiner system.
\end{proposition}

\noindent
The most common specific constructions of finite affine planes of order $q$ have as points the elements of $\FF^2$ for a finite field $\FF$ with $|\FF| = q$, and as lines the one-dimensional affine subspaces in this vector space over $\FF$.
Analogously, the most common specific constructions of finite projective planes of order $q$ have as points the one-dimensional linear subspaces of $\FF^{3}$, and as blocks the two-dimensional linear subspaces of $\FF^{3}$.
These give finite projective and affine planes of any order $q$ equal to a prime power; the question of whether there exist finite projective or affine planes of other orders is a longstanding open problem in combinatorics.

The other ingredient in the Steiner ETF construction is the following class of matrices, which describe certain structured orthonormal bases of $\CC^N$.
\begin{definition}
    A \emph{(complex) Hadamard matrix} is a matrix $\bH \in \CC^{N \times N}$ such that $|H_{ij}| = 1$ for all $i, j \in [N]$, and $\bH\bH^* = N \bm I_N$.
\end{definition}

\subsubsection{Equiangular Tight Frames: General Theory}

We next review some background on the theory of ETFs.
ETFs do not exist in all pairs of dimension $(N, r)$, and the known constructions are mostly based on exceptional combinatorial structures.
Determining for which dimensions ETFs do exist is thus an important open problem with implications ranging many topics in combinatorics.
Real-valued ETFs are better understood than the more general complex-valued ETFs, and the techniques for treating these two variants are often somewhat different.
In particular, there are several correspondences between real ETFs and strongly regular graphs \cite{fickus:15,fickus:16:centroidal} which allow combinatorial constructions to immediately yield real ETFs, while no such general connections are known to hold for complex ETFs.
More comprehensive references on these aspects of the theory of ETFs include \cite{tropp:07,casazza:08,fickus:15:tables}.

The following important result shows that ETFs are extremal among UNTFs in the sense of \emph{worst-case coherence}.
Moreover, when an ETF exists, $\alpha$ is determined by $N$ and $r$.
\begin{proposition}[Welch Bound \cite{welch:74}]
    \label{prop:etf-welch-bound}
    If $\bv_1, \dots, \bv_N \in \CC^{r}$ with $\|\bv_i\|_2 = 1$, then
    \begin{equation}
        \max_{1 \leq i, j \leq N \atop i \neq j} |\la \bv_i, \bv_j \ra| \geq \sqrt{\frac{N - r}{r(N - 1)}},
    \end{equation}
    with equality if and only if $\bv_1, \dots, \bv_N$ form an ETF.
\end{proposition}

The other important limitation on ETFs that we will be concerned with is on the maximum number of vectors in $\RR^r$ or $\CC^r$ that can form an ETF.
We include the proof of this result, since it is short and involves a matrix associated to an ETF that will play an important role in our later calculations.
\begin{proposition}[Gerzon Bound \cite{lemmens:91}]
    \label{prop:etf-gerzon-bound}
    Let $\FF \in \{ \RR, \CC \}$, and let $\bv_1, \dots, \bv_N \in \FF^r$ form an ETF with coherence $\alpha < 1$.
    Then,
    \begin{equation}
        N \leq \left\{\begin{array}{ccl} r(r + 1) / 2 & : & \FF = \RR \\ r^2 & : & \FF = \CC \end{array}\right\}.
    \end{equation}
\end{proposition}
\begin{proof}
    Let $\MM = \RR^{r \times r}_{\sym}$ if $\FF = \RR$ and $\MM = \CC^{r \times r}_{\herm}$ if $\FF = \CC$.
    For all $i \neq j$, $\la \bv_i\bv_i^*, \bv_j\bv_j^* \ra = \alpha^2$.
    Thus, the Gram matrix of the $\bv_i\bv_i^* \in \MM$ is
    \begin{equation}
        |\bX|^{\odot 2} = (1 - \alpha^2) \bm I_N + \alpha^2 \one\one^\top,
    \end{equation}
    which is non-singular since $\alpha < 1$.
    The $\bv_i\bv_i^*$ are then linearly independent, so $N \leq \dim(\MM)$, and the result follows.
\end{proof}

Lastly, an involution of ETFs called the \emph{Naimark complement} will play an important role in some of our reasoning.
In the sequel, we will sometimes identify an ETF $\bv_1, \dots, \bv_N \in \CC^r$ with its so-called \emph{synthesis matrix} $\bV \in \CC^{r \times N}$, the matrix whose columns are the $\bv_i$.
Since the rows of $\bV$ are orthogonal and have equal norm, $\bV$ may be completed by further rows to form a scaled orthogonal matrix.
The added rows form a matrix, which we denote $\bV^\prime \in \CC^{(N - r) \times N}$, which also has orthogonal rows of equal norm and columns of equal norm, and hence, suitably scaled, is the synthesis matrix of a new UNTF consisting of $N$ vectors in $\CC^{N - r}$.
Moreover, one may verify that this UNTF is in fact another ETF (with different coherence).
Choosing among the possible completions to an orthogonal matrix appropriately, we may guarantee that the Naimark complement is an involution of ETFs: $\bV^{\prime\prime} = \bV$ for every ETF $\bV$.

\subsubsection{Equiangular Tight Frames: Sparsity and Spark}

Two quantities of interest for ETFs are the \emph{sparsity}, the number of non-zero entries of the ETF vectors or the synthesis matrix, and the \emph{spark}, the smallest number of ETF vectors involved in a non-trivial linear dependency.
The conventional definition of the latter is
\begin{equation}
    \mathsf{spark}(\bV) \colonequals \min_{\substack{\bx \in \RR^N \setminus \{\bm 0\} \\ \bV\bx = \bm 0}} \|\bx\|_0 = \min_{\bx \in \row(\bV)^{\perp} \setminus \{ \bm 0\}} \|\bx\|_0.
\end{equation}
The natural dual measure of sparsity, sometimes called \emph{cospark}, is
\begin{equation}
    \mathsf{sparsity}(\bV) \colonequals \min_{\bx \in \row(\bV) \setminus \{ \bm 0\}} \|\bx\|_0.
\end{equation}
This is a stronger notion that merely the sparsity of the synthesis matrix $\bV$ itself, but our results are most naturally stated in these terms.
One convenience of working with this notion of sparsity is that, since the Naimark complement exchanges the row space and kernel of an ETF synthesis matrix, our sparsity is in fact merely the spark of the Naimark complement:
\begin{equation}
    \label{eq:sparsity-spark-naimark}
    \mathsf{sparsity}(\bV) = \mathsf{spark}(\bV^\prime).
\end{equation}

We will later be concerned with the quality of certain lower bounds on the spark (and thus, via the above remark, the sparsity) of ETFs, so we briefly review the existing general lower bounds we are aware of.
The simplest spark lower bound is based on an argument via the Gershgorin circle theorem \cite{donoho:03}, which we present below.
\begin{proposition}
    \label{prop:gershgorin-spark-bound}
    If $\bV$ is an ETF with coherence $\alpha$, then $\mathsf{spark}(\bV) \geq 1 + \alpha^{-1}$.
\end{proposition}
\begin{proof}
    One may easily check the following alternative formulation of the spark:
    \[ \mathsf{Spark}(\bV) = 1 + \max\{k: \text{every } k \times k \text{ principal minor of } \bV^\top\bV \text{ is non-singular}\}. \]
    It then suffices to show that every $\alpha^{-1} \times \alpha^{-1}$ principal minor is non-singular.
    Such a minor is a $\alpha^{-1} \times \alpha^{-1}$ matrix whose diagonal entries are 1 and whose off-diagonal entries are $\alpha$ in magnitude.
    By the Gershgorin circle theorem, the eigenvalues of such a matrix are all at least $1 - \frac{\alpha^{-1} - 1}{\alpha^{-1}} > 0$, hence it is non-singular.
\end{proof}
\noindent
Another lower bound on the spark follows from considerations of so-called \emph{numerically erasure-robust frames (NERFs)} in \cite{fickus:12:nerfs}.\footnote{We thank Dustin Mixon for bringing this connection to our attention.}
\begin{proposition}
    \label{prop:nerf-spark-bound}
    If $\bV$ is an ETF of $N$ vectors in $\RR^r$, then
    \begin{equation}
        \mathsf{spark}(\bV) \geq N \left(1 + \frac{(N - r)(N - r - 1)}{N - 1}\right)^{-1}.
    \end{equation}
\end{proposition}
\begin{proof}
    This follows directly from Theorem 5 of \cite{fickus:12:nerfs}.
\end{proof}
Our result on spark will sharpen Proposition~\ref{prop:nerf-spark-bound}, so we will be interested in cases where this result is superior to the Gershgorin circle argument.

To this end, consider a scaling regime where $N - r \sim r^{\beta}$.
Then, the NERF spark lower bound scales as $r^{2 - 2\beta}$, while the Gershgorin circle spark lower bound scales as $r^{1 - \beta / 2}$.
These are asymptotically equal at $\beta = \frac{2}{3}$, and thus for any $\beta < \frac{2}{3}$, we expect the NERF bound to be asymptotically superior, and for $\beta > \frac{2}{3}$ expect the Gershgorin circles bound to be asymptotically superior.
Equivalently, we expect the NERF bound to be asymptotically superior on the Naimark complements of ETFs where $N \sim r^\beta$ for $\beta > \frac{3}{2}$.

There are few known infinite families of ETFs where $N$ scales super-linearly with $r$.
One family of examples is given by the maximal ETFs that saturate the Gerzon bound, Proposition~\ref{prop:etf-gerzon-bound}, for which $N \sim r^2$, but it is perhaps the most prominent open problem in the theory of ETFs to determine in what dimensions (in both the real and complex cases) maximal ETFs exist, and it is in particular unknown if infinitely many maximal ETFs exist.
Curiously, to the best of our knowledge all other known families with super-linear scaling in fact scale precisely as $N \sim r^{3/2}$.
(We are aware of such families based on difference sets in finite abelian groups \cite{ding:07}, Steiner systems corresponding to finite affine and projective planes \cite{fickus:12:steiner}, hyperovals in finite projective planes \cite{fickus:16:ovals}, and abelian generalized quadrangles \cite{fickus:17}.)
On these examples, the bound we will present gives an improvement of sub-leading order on the NERF bound.

\subsubsection{Steiner Equiangular Tight Frames}
\label{sec:steiner}

Finally, we present the Steiner ETF construction of \cite{fickus:12:steiner}, on which we will illustrate our results in greatest detail.
This construction is based on combining two types of combinatorial objects, Steiner systems and Hadamard matrices, as defined in Section~\ref{sec:background:comb}
\begin{proposition}[Theorem 1 of \cite{fickus:12:steiner}]
    \label{prop:steiner-etfs}
    Let $(S, \sB)$ be a $(2, k, v)$-Steiner system, let $b = \frac{v(v - 1)}{k(k - 1)}$ be the number of blocks, and let $\rho = \frac{v - 1}{k - 1}$ be the number of blocks containing any point.
    Let $\bH \in \CC^{(1 + \rho) \times (1 + \rho)}$ be a complex Hadamard matrix.
    Then, there exists an ETF of $N = \rho v$ vectors in $\CC^r$ with $r = b$.

    A matrix $\bV \in \CC^{b \times \rho v}$ whose columns are the ETF vectors may be constructed as follows.
    Let $\bN \in \{0, 1\}^{v \times b}$ be the incidence matrix of the Steiner system.
    For each $j = 1, \dots, v$, let $\bV_j \in \CC^{b \times \rho}$ be formed from the $j$th column of $\bN^\top$ by replacing every entry equal to 1 with a distinct row of $\bH$, and every entry equal to 0 with a row of zeros.
    Then, let $\bV \colonequals \rho^{-1/2}[\bV_1 \cdots \bV_v]$.
    If $\bH$ is a real-valued Hadamard matrix, then the resulting ETF is also real.
\end{proposition}

\section{Projecting to the Perturbation Subspace}

Returning to the more general setting of the introduction, let $\bv_1, \dots, \bv_N \in \CC^r$ form a UNTF, let $\bV \in \CC^{r \times N}$ have the $\bv_i$ as its columns, and let $\bX = \bV^* \bV$ be the Gram matrix.
In this section, we compute the orthogonal projection operator to $\pert_{\widetilde{\fE}^N}(\bX)$,
\begin{equation}
    \sP\bA \colonequals \argmin_{\bB \in \pert_{\widetilde{\fE}^N}(\bX)} \frac{1}{2}\|\bA - \bB \|_F^2.
    \label{eq:proj-var}
\end{equation}
The key tool is the following classical result describing perturbation subspaces for the elliptope.
\begin{proposition}[Theorem 1(a) of \cite{li:94}]
    \label{prop:li-tam}
    Let $\bX \in \widetilde{\fE}^N$ with $\rank(\bX) = r$, and let $\bv_1, \dots, \bv_N \in \CC^r$ such that $X_{ij} = \la \bv_i, \bv_j \ra$ for $i, j \in [N]$.
    Write $\bV \in \CC^{r \times N}$ for the matrix with the $\bv_i$ as its columns.
    Then,
    \begin{equation}
      \mathsf{pert}_{\widetilde{\fE}^N}(\bX) = \left\{ \bV^* \bH \bV: \bH \in \CC_\herm^{r \times r}, \bv_i^* \bH \bv_i = 0 \text{ for all } i \in [N] \right\}.
  \end{equation}
\end{proposition}
Our formula is then the following (retaining the notation from Proposition~\ref{prop:li-tam}).

\begin{lemma}
    \label{lem:proj}
    Let $\bv_1, \dots, \bv_N \in \CC^r$ form a UNTF.
    Suppose the $\bv_i\bv_i^*$ are linearly independent, or equivalently that $|\bX|^{\odot 2}$ is non-singular.
    Let $\bx_1, \dots, \bx_N$ be the columns of $\bX$.
    Then,
    \begin{equation}
        \sP\bA = \frac{r^2}{N^2}\Bigg(\bX\bA \bX - \sum_{i, j = 1}^N (|\bX|^{\odot 2})^{-1}_{ij} (\bx_i^* \bA \bx_i) \bx_j\bx_j^*\Bigg).
        \label{eq:proj-formula}
    \end{equation}
\end{lemma}
(In particular, all ETFs except trivial ones with $r = 1$ satisfy the hypotheses.)

\begin{proof}
By Proposition~\ref{prop:li-tam} and the variational characterization \eqref{eq:proj-var}, we have
\begin{align}
  \sP\bA &= \bV^* \bH^\star(\bA) \bV, \label{eq:P-formula-inter} \\
  \bH^\star(\bA) &\colonequals \argmin_{\substack{\bH \in \CC_{\herm}^{r \times r} \\ \bv_i^*\bH\bv_i = 0 \text{ for } i \in [N]}} \mathsf{obj}(\bH; \bA), \\
  \mathsf{obj}(\bH; \bA) &\colonequals \frac{1}{2}\|\bA - \bV^* \bH \bV\|_F^2 \nonumber \\
  &= \frac{1}{2}\|\bA\|_F^2 + \frac{N^2}{2r^2}\|\bH\|_F^2 - \la \bV\bA\bV^*, \bH \ra.
\end{align}
(In the final equation we use the UNTF property.)
Introducing a vector of Lagrange multipliers $\bm\gamma \in \RR^N$ for the constraints in the optimization defining $\bH^\star$, we obtain the Lagrangian
\begin{align}
  L(\bH, \bm\gamma; \bA)
  &\colonequals \mathsf{obj}(\bH; \bA) - \sum_{i = 1}^N \gamma_i \bv_i^*\bH\bv_i \nonumber \\
  &= \frac{1}{2}\|\bA\|_F^2 + \frac{N^2}{2r^2}\|\bH\|_F^2 - \left\la \bV\bA\bV^* + \sum_{i = 1}^N \gamma_i\bv_i\bv_i^*, \bH \right\ra.
\end{align}
The first-order condition for optimality then implies that
\begin{equation}
    \bH^\star(\bA) = \bV\bA\bV^* + \sum_{i = 1}^N \gamma_i(\bA) \bv_i\bv_i^*
    \label{eq:h-star-formula}
\end{equation}
for some $\bm\gamma(\bA)$ such that $\bv_i^* \bH^\star(\bA) \bv_i = 0$ for $i \in [N]$.
These constraints may be written as the system
\begin{equation}
    \sum_{j = 1}^N |\la \bv_i, \bv_j\ra|^2\gamma_j(\bA) = -\bv_i^*\bV\bA\bV^*\bv_i \text{ for } i \in [N],
\end{equation}
which is a linear system in $\bm\gamma(\bA)$ with matrix $|\bX|^{\odot 2}$.
Since this matrix is invertible by assumption, there is a unique solution
\begin{equation}
    \gamma_i(\bA) = -\sum_{j = 1}^N (|\bX|^{\odot 2})^{-1}_{ij}\bv_j^*\bV\bA\bV^*\bv_j.
\end{equation}
The result follows by substituting into \eqref{eq:h-star-formula} and then \eqref{eq:P-formula-inter}.
\end{proof}

The result we will use to obtain sparsity inequalities follows from manipulations of the fact that $\sP$ is psd, whereby $\la \bA, \sP\bA \ra \geq 0$ for any $\bA \in \CC^{N \times N}_{\herm}$.
Surprisingly, this fact is equivalent to the following simpler matrix inequality.
\begin{lemma}
    \label{lem:ineq}
    Let $\bv_1, \dots, \bv_N \in \CC^r$ form a UNTF.
    Suppose the $\bv_i\bv_i^*$ are linearly independent, or equivalently that $|\bX|^{\odot 2}$ is non-singular.
    Then,
    \begin{equation}
        (|\bV|^{\odot 2})(|\bX|^{\odot 2})^{-1}(|\bV|^{\odot 2})^\top \preceq \bm I_r,
        \label{eq:untf-overlap-ineq}
    \end{equation}
    or equivalently
    \begin{equation}
        (|\bV|^{\odot 2})^\top(|\bV|^{\odot 2}) \preceq |\bX|^{\odot 2}.
        \label{eq:untf-overlap-ineq-2}
    \end{equation}
\end{lemma}

\begin{remark}
    After obtaining Lemma~\ref{lem:ineq}, we discovered the reference \cite{wang:97}, which gives a general result of the form \eqref{eq:untf-overlap-ineq-2} not depending on the $\bv_i$ forming a tight frame.
    The proof of \cite{wang:97} builds a psd block matrix through the Schur product theorem and uses that its Schur complement remains psd.
    Our proof gives a more direct geometric argument, and, as we discuss in Section~\ref{sec:problems}, may generalize to higher-degree inequalities.
\end{remark}

\begin{proof}
    Let us write
    \begin{align}
      \sP\bA &\colonequals \sP_1\bA - \sP_2\bA, \\
      \sP_1\bA &\colonequals \frac{r^2}{N^2}\bX \bA \bX, \\
      \sP_2\bA &\colonequals \frac{r^2}{N^2}\sum_{i, j = 1}^N (|\bX|^{\odot 2})^{-1}_{ij} (\bx_i^* \bA \bx_i) \bx_j\bx_j^*.
    \end{align}
    Writing $\what{\bV} \colonequals \sqrt{\frac{r}{N}}\bV$, we have
    \begin{align}
      \la \bA, \sP_1\bA \ra &= \left\|\what{\bV}\bA\what{\bV}^* \right\|_F^2, \\
      \la \bA, \sP_2\bA \ra &= \sum_{i,j = 1}^N (|\bX|^{\odot 2})^{-1}_{ij} \left(\bv_i^*\what{\bV} \bA \what{\bV}^* \bv_i\right)\left(\bv_j^* \what{\bV} \bA \what{\bV}^* \bv_j\right).
    \end{align}
Any $\bH \in \CC^{r \times r}_{\herm}$ may be expressed in the form $\frac{r}{N}\what{\bV}\bA\what{\bV}^*$ by taking $\bA = \what{\bV}^*\bH\what{\bV}$.
Therefore, the inequality $\la \bA, \sP \bA \ra = \la \bA, \sP_1 \bA \ra - \la \bA, \sP_2 \bA \ra \geq 0$ holding for $\bA \in \CC^{N \times N}_{\herm}$ is equivalent to the following inequality holding for $\bH \in \CC^{r \times r}_{\herm}$:
\begin{equation}
    \sum_{i = 1}^N \sum_{j = 1}^N (|\bX|^{\odot 2})^{-1}_{ij} \left(\bv_i^* \bH \bv_i\right)\left(\bv_j^* \bH \bv_j\right) \leq \|\bH\|_F^2.
    \label{eq:ineq-H}
\end{equation}
Moreover, since applying a unitary transformation to a UNTF produces another UNTF having the same Gram matrix $\bX$, we may assume that $\bH$ is diagonal and real-valued, $\bH = \diag(\bm\lambda)$ with $\bm\lambda \in \RR^r$.
Rewriting the resulting inequality as an inequality of quadratic forms in $\bm\lambda$ then gives the result.
The equivalence of \eqref{eq:untf-overlap-ineq} and \eqref{eq:untf-overlap-ineq-2} is a general fact; see e.g.\ \cite{baksalary:91}.
\end{proof}

\section{Sparsity Inequalities for ETFs}

Our remaining results take advantage of the fact that the matrix $|\bX|^{\odot 2}$ is very simple for an ETF:
\begin{equation}
    |\bX|^{\odot 2} = (1 - \alpha^2)\bm I_r + \alpha^2 \one\one^\top.
    \label{eq:x-sq-etf}
\end{equation}
Moreover, by the Welch bound (our Proposition~\ref{prop:etf-welch-bound}), $\alpha$ depends only on the dimension parameters $N$ and $r$:
\begin{equation}
    \alpha = \sqrt{\frac{N - r}{r(N - 1)}}.
    \label{eq:alpha-etf}
\end{equation}
In this case, Lemma~\ref{lem:ineq} gives the following.

\begin{theorem}
    \label{thm:etf-ineq}
    Let $\bv_1, \dots, \bv_N \in \CC^r$ form an ETF.
    Define $\bR \colonequals (|\bV|^{\odot 2})(|\bV|^{\odot 2})^\top \in \RR^{r \times r}$, with entries $R_{k \ell} = \sum_{i = 1}^N |(\bv_i)_k|^2|(\bv_i)_\ell|^2$.
    Then,
    \begin{equation}
        \bm R \preceq \frac{1 - \frac{1}{r}}{1 - \frac{1}{N}}\bm I_r + \frac{\frac{N}{r} - 1}{r(1 - \frac{1}{N})}\one\one^\top.
        \label{eq:etf-ineq}
    \end{equation}
\end{theorem}

\begin{proof}
    From \eqref{eq:x-sq-etf} and \eqref{eq:alpha-etf}, we find
    \begin{equation}
        (|\bX|^{\odot 2})^{-1} = \frac{1 - \frac{1}{N}}{1 - \frac{1}{r}}\left(\bm I_N - \frac{N - r}{N(N - 1)}\one\one^\top\right).
    \end{equation}
    We have
    \begin{equation}
        (|\bV|^{\odot 2})\one\one^\top(|\bV|^{\odot 2})^\top = \frac{N^2}{r^2}\one\one^\top,
    \end{equation}
    and substituting into \eqref{eq:untf-overlap-ineq} gives the result.
\end{proof}

The next result gives the exact dimension of the subspace on which the inequality \eqref{eq:etf-ineq} is sharp on the Steiner ETFs described in Proposition~\ref{prop:steiner-etfs} (and in greater detail in the original work \cite{fickus:12:steiner}).

\begin{proposition}
    Let $\bv_1, \dots, \bv_N \in \RR^r$ be a Steiner ETF constructed from a $(2, k, v)$-Steiner system and a Hadamard matrix of suitable size.
    Let $\bR \in \RR^{r \times r}$ have entries $R_{k \ell} = \sum_{i = 1}^N |(\bv_i)_k|^2|(\bv_i)_\ell|^2$.
    Then,
    \begin{equation}
        \dim\left(\ker\left(\frac{1 - \frac{1}{r}}{1 - \frac{1}{N}}\bm I_r + \frac{\frac{N}{r} - 1}{r(1 - \frac{1}{N})}\one\one^\top - \bR\right)\right) = v.
    \end{equation}
    When the Steiner system is a finite projective plane, then the inequality \eqref{eq:etf-ineq} is an equality of matrices.
\end{proposition}
\begin{proof}
    Let $b = \frac{v(v - 1)}{k(k - 1)}$ be the number of blocks in the underlying Steiner system and $\rho = \frac{v - 1}{k - 1}$ be the number of blocks in which each point lies.
    Then, per the construction of \cite{fickus:12:steiner} described in Proposition~\ref{prop:steiner-etfs}, $r = b$ and $N = v(1 + \rho)$.

    We first compute $\bR$: letting $\bN \in \RR^{v \times b}$ be the incidence matrix of points and blocks of the Steiner system, we find
    \begin{equation}
        \bR = \frac{\rho + 1}{\rho^2}\bN^\top \bN = k \bm I_b + \bA_G,
    \end{equation}
    where $\bA_G$ is the adjacency matrix of the block intersection graph $G$ of the Steiner system.
    $G$ is a strongly regular graph \cite{goethals:91}, admitting a spectral expansion
    \begin{equation}
        \bA_G = k(\rho - 1)\what{\one}_b\what{\one}_b^\top + (\rho - 1 - k)\bP_{U_+} - k\bP_{U_-},
    \end{equation}
    where $\what{\one}_b = \frac{1}{\sqrt{b}}\one_b$, $U_{\pm}$ are eigenspaces orthogonal to one another and to the vector $\one$ and satisfying $U_+ \oplus U_- \oplus \one = \RR^b$, and $\bP_{U_{\pm}}$ are the projectors onto these subspaces.
    The corresponding dimensions are
    \begin{align}
      \dim(U_+) &= v - 1, \label{eq:bibd-srg-pos-dim} \\
      \dim(U_-) &= b - v.
    \end{align}
    We thus obtain the spectral expansion of $\bR$,
    \begin{equation}
        \bR = k\left(1 + \frac{1}{\rho}\right)\what{\one}_r\what{\one}_r^\top + \left(1 - \frac{1}{\rho^2}\right)\bP_{U_+}.
    \end{equation}
    Some algebraic manipulations show that the following identities hold between the eigenvalues of the right-hand side of \eqref{eq:etf-ineq} and those of $\bR$:
    \begin{align}
      k\left(1 + \frac{1}{\rho}\right) &= \frac{1 - \frac{1}{r}}{1 - \frac{1}{N}} + \frac{N - r}{r(1 - \frac{1}{N})}, \\
      1 - \frac{1}{\rho^2} &= \frac{1 - \frac{1}{r}}{1 - \frac{1}{N}},
    \end{align}
    and therefore in fact
    \begin{equation}
        \frac{1 - \frac{1}{r}}{1 - \frac{1}{N}}\bm I_r + \frac{\frac{N}{r} - 1}{r(1 - \frac{1}{N})}\one\one^\top - \bm R = \frac{1 - \frac{1}{r}}{1 - \frac{1}{N}}\bP_{U_-}.
    \end{equation}
    The first part of the result then follows from the dimension formula \eqref{eq:bibd-srg-pos-dim}.
    For the special case of finite projective planes, it suffices to note that in this case $r = b = v$, which is easily verified from the formulae in Proposition~\ref{prop:finite-planes}
\end{proof}

Finally, to illustrate how some more concrete results may be obtained from Theorem~\ref{thm:etf-ineq}, we give corollaries controlling ETF sparsity, spark, and the overlap of rows of the synthesis matrix $\bV$.
Recall that the \emph{spark} is defined as
\begin{equation}
    \mathsf{spark}(\bV) \colonequals \min_{\substack{\bx \in \RR^N \setminus \{\bm 0\} \\ \bV^\top\bx = \bm 0}} \|\bx\|_0 = \min_{\bx \in \row(\bV)^{\perp} \setminus \{ \bm 0\}} \|\bx\|_0.
\end{equation}
The natural dual measure of sparsity, sometimes called \emph{cospark}, is
\begin{equation}
    \mathsf{sparsity}(\bV) \colonequals \min_{\bx \in \row(\bV) \setminus \{ \bm 0\}} \|\bx\|_0,
\end{equation}
which gives control of the sparsity of the entire matrix $\bV$ by controlling each row.

\begin{corollary}
    \label{corr:etf-spark}
    Let $\bv_1, \dots, \bv_N \in \CC^r$ form an ETF, and let $\bV \in \CC^{r \times N}$ have the $\bv_i$ as its columns.
    Then,
    \begin{align}
      \mathsf{sparsity}(\bV) &\geq N\left(1 + \frac{(r - 1)^2}{N - 1}\right)^{-1}, \\
      \mathsf{spark}(\bV) &\geq N\left(1 + \frac{(N - r - 1)^2}{N - 1}\right)^{-1}.
    \end{align}
\end{corollary}
\begin{proof}
    Let $\by \in \CC^r$ with $\|\by\|_2 = 1$.
    Let $\bU \in \CC^{r \times r}$ be a unitary matrix whose first column is $\by$.
    We apply \eqref{eq:etf-ineq} to the ETF with synthesis matrix $\bU\bV$, whose first row is $\by^*\bV$.
    Comparing the entries in the upper left corners of the matrices on either side of the inequality, we find
    \begin{equation}
        \|\bV^*\by\|_4^4 \leq \frac{1 - \frac{1}{r}}{1 - \frac{1}{N}} + \frac{\frac{N}{r} - 1}{r(1 - \frac{1}{N})} = \frac{\frac{N}{r} + r - 2}{r(1 - \frac{1}{N})}.
        \label{eq:spark-1}
    \end{equation}
    Writing $\bx = \bV^*\by$, note that in general, by the Cauchy-Schwarz inequality,
    \begin{equation}
        \|\bx\|_2^4 = \left(\sum_{i = 1}^N |x_i|^2\right)^2 \leq \left(\sum_{i = 1}^N |x_i|^4\right)\left(\sum_{i = 1}^N \mathbb{1}\{x_i \neq 0\}\right) = \|\bx\|_4^4 \cdot \|\bx\|_0.
    \end{equation}
    In our case, $\|\bx\|_2^2 = \by^*\bV\bV^*\by = \frac{N}{r}$, and thus
    \begin{equation}
        \|\bV^*\by\|_0 \geq \frac{\|\bV^*\by\|_2^4}{\|\bV^*\by\|_4^4} \geq \frac{N^2}{r^2} \cdot \frac{r(1 - \frac{1}{N})}{\frac{N}{r} + r - 2} = N\left(1 + \frac{(r - 1)^2}{N - 1}\right)^{-1}.
    \end{equation}
    This gives the sparsity result, and the spark result follows from the same argument applied to the Naimark complement by \eqref{eq:sparsity-spark-naimark}.
\end{proof}
We compare this bound to the NERF and Gershgorin bounds of Proposition~\ref{prop:gershgorin-spark-bound} and Proposition~\ref{prop:nerf-spark-bound} respectively in Table~\ref{tab:spark-bounds}, and find that on the Naimark complements of infinite families of ETFs with dimensions scaling as $N \sim r^{3/2}$, our bound gives an improvement of sub-leading order on the NERF bound.
On the other hand, both our bound and the NERF bound are incomparable in general to the Gershgorin circle bound on such ETFs, sometimes being superior and sometimes inferior depending on the specific construction.

The final corollary we mention shows that the overlap between sparsity patterns of distinct rows of an ETF in fact has a certain ``typical'' value for a given pair of dimensions $r$ and $N$, from which its possible deviations are bounded.
\begin{corollary}
    Let $\bv_1, \dots, \bv_N \in \CC^r$ form an ETF, and let $\bV \in \RR^{r \times N}$ have the $\bv_i$ as its columns.
    Let $\ba, \bb \in \row(\bV)$ with $\la \ba, \bb \ra = 0$ and $\|\ba\|_2^2 = \|\bb\|_2^2 = \frac{N}{r}$ (for instance, two distinct rows of $\bV$).
    Let $D \colonequals \frac{N}{r^2}(1 + \frac{(r - 1)^2}{N - 1})$ and $E \colonequals \frac{\frac{N}{r} - 1}{r(1 - \frac{1}{N})}$.
    Then, $D \geq \|\ba\|_4^4$ and $D \geq \|\bb\|_4^4$, and
    \begin{equation}
        \left| \la |\ba|^{\odot 2}, |\bb|^{\odot 2} \ra - E \right|^2 \leq \left(D - \|\ba\|_4^4\right)\left(D - \|\bb\|_4^4\right).
    \end{equation}
\end{corollary}
\begin{proof}
    Note that $D$ is equal to the diagonal entries and $E$ to the off-diagonal entries of the right-hand side of \eqref{eq:etf-ineq}.
    Thus, $D \geq \|\ba\|_4^4$ and $D \geq \|\bb\|_4^4$ follows by the previous argument of Corollary~\ref{corr:etf-spark}.
    As in that argument, we may without loss of generality assume $\ba$ and $\bb$ occur as the first two rows of $\bV$.
    Then, the relation between the upper left $2 \times 2$ minors of either side of \eqref{eq:etf-ineq} is
    \begin{equation}
        \left[\begin{array}{cc} \|\ba\|_4^4 & \la |\ba|^{\odot 2}, |\bb|^{\odot 2} \ra \\ \la |\ba|^{\odot 2}, |\bb|^{\odot 2} \ra & \|\bb\|_4^4 \end{array}\right] \preceq \left[\begin{array}{cc} D & E \\ E & D \end{array}\right].
    \end{equation}
    Taking the determinant of the difference of the right- and left-hand sides of the above then gives the result.
\end{proof}

\begin{table}
    \centering
    \small
    \begin{tabular}{llllll}
      \hline
      \textbf{Construction} & $\bm N$ & $\bm r$ & \textbf{Gershgorin} & \textbf{NERF} & \textbf{Our Bound} \\
      \hline
      Steiner: Affine \cite{fickus:12:steiner} & $q^3 + 2q^2$ & $q^2 + q$ & $q^2 + q$ & $q^2 + q - 1$ & $q^2 + q$ \\
      Steiner: Projective \cite{fickus:12:steiner} & $q^3 + 3q^2 + 3q + 2$ & $q^2 + q + 1$ & $q^2 + 2q + 2$ & $q^2 + 3q + 1$ & $q^2 + 3q + 2$ \\
      Polyphase BIBD \cite{fickus:17} & $q^3 + 1$ & $q^2 - q + 1$ & $q^2 + 1$ & $q^2 + q - 1$ & $q^2 + q$ \\
      Hyperovals \cite{fickus:16:ovals} & $q^3 + q^2 - q$ & $q^2 + q - 1$ & $q^2$ & $q^2 - q + 3$ & $q^2 - q + 4$ \\
      \hline
    \end{tabular}
    \caption{\textbf{Spark Lower Bound Comparison.} We tabulate three lower bounds on the spark of ETFs of $N$ vectors in $\RR^r$ formed as the Naimark complements of ETFs belonging to infinite families with $N \sim r^{3 / 2}$.
      We express the parameters of these Naimark complements in terms of an integer parameter $q$ for each family.
      (There is a fifth construction based on difference sets in finite abelian groups having this scaling \cite{ding:07}, but its parameters are the same as those of Steiner ETFs built from finite affine planes, so we omit it from the table.)
      The first two bounds are from prior work described in Propositions~\ref{prop:gershgorin-spark-bound} and \ref{prop:nerf-spark-bound}, and the last is from our Corollary~\ref{corr:etf-spark}.
      Since the spark is always an integer, the ceiling function may be applied to any valid lower bound on it to obtain a stronger lower bound.
      Thus, when the bound expressions do not simplify to integer values in terms of $q$, we take a partial fraction decomposition and give a bound holding for sufficiently large $q$.}
    \label{tab:spark-bounds}
\end{table}

\section{Open Problems}

\label{sec:problems}

Returning to the more general family of relaxations of the cut polytope described in Definition~\ref{def:pm-matrix}, the main question we propose for investigation is the following.

\begin{question}
    Given the Gram matrix of a real ETF of $N$ vectors in $\RR^r$, what is the largest $d$ for which it belongs to $\fE_d^N$?
    Does the answer depend only on $N$ and $r$?
\end{question}

\noindent
More specifically, we are interested in the details of the construction that would underlie such a result.
\begin{question}
    When $\bX$ is the Gram matrix of a real ETF and $\bX \in \fE_d^N$ with a ``witness'' $\bY$ per Definition~\ref{def:pm-matrix}, is there a tractable description of the eigenspaces of $\bY$?
\end{question}
\noindent
If this is the case, then we may hope to imitate the present approach: compute projectors to the analogous subspaces for complex ETFs, write down the positivity relation for these operators, and derive polynomial inequalities in the ETF entries.
It remains to be seen, however, whether such inequalities could still be interpreted as giving information about sparsity of ETFs, or whether they would control new quantities for higher degrees.

\bibliographystyle{plain}
\bibliography{main}

\end{document}